\documentclass[12pt]{amsart}
\usepackage{amsfonts,amssymb,amscd,amstext,mathrsfs}
\usepackage[utf8]{inputenc}
\usepackage{hyperref}
\usepackage{verbatim}

\usepackage{times}
\usepackage{enumerate}
\usepackage[up,bf]{caption}

\input xy
\xyoption{all}

\newtheorem{theorem}{Theorem}[section]

\newtheorem{lemma}[theorem]{Lemma}
\newtheorem{corollary}[theorem]{Corollary}

\theoremstyle{definition}

\newtheorem{remark}[theorem]{Remark}

\numberwithin{equation}{section}
\numberwithin{figure}{section}

\newcommand\Qcal{\mathcal{Q}}

\newcommand\be{\mathbf{e}}
\newcommand\bp{\mathbf{p}}

\newcommand\bv{\mathbf{v}}
\newcommand\bx{\mathbf{x}}

\newcommand\bz{\mathbf{z}}

\newcommand\Cscr{\mathscr{C}}

\newcommand\C{\mathbb{C}}
\newcommand\D{\overline{\mathbb D}}

\renewcommand\D{\mathbb D}

\newcommand\R{\mathbb{R}}

\newcommand\igot{\mathfrak{i}}

\renewcommand\igot{\mathfrak{i}}

\renewcommand\imath{\igot}
\newcommand\zero{\mathbf{0}}

\newcommand\wt{\widetilde}

\newcommand\di{\partial}
\newcommand\dibar{\overline\partial}

\newcommand\dist{\mathrm{dist}}
\renewcommand\span{\mathrm{span}}

\newcommand\tr{\mathrm{tr}}
\newcommand\Hess{\mathrm{Hess}}

\def\dist{\mathrm{dist}}
\def\span{\mathrm{span}}

\numberwithin{equation}{section}

\begin{document}
\title{Every smoothly bounded $p$-convex domain in $\R^n$ admits a \\
$p$-plurisubharmonic defining function}
\author{Franc Forstneri{\v c}}

\address{Franc Forstneri\v c, Faculty of Mathematics and Physics, University of Ljubljana, and Institute of Mathematics, Physics, and Mechanics, Jadranska 19, 1000 Ljubljana, Slovenia}
\email{franc.forstneric@fmf.uni-lj.si}

\thanks{The author is supported by research program P1-0291 and grant J1-3005 from ARRS, Republic of Slovenia.}


\date{December 1, 2021}


\begin{abstract}
We show that every bounded domain $D$ in $\R^n$ 
with smooth $p$-convex boundary for $2\le p < n$ admits a smooth defining function $\rho$ which is 
$p$-plurisubharmonic on $\overline D$; if in addition $bD$ has no $p$-flat points then $\rho$ can be
chosen strongly $p$-plurisubharmonic on $D$. If $bD$ is $2$-convex then 
for any open connected conformal surface $M$ and conformal harmonic map 
$f:M\to \overline D$, either $f(M)\subset D$ or $f(M)\subset bD$.
In particular, every conformal harmonic map $\D^*\to D$ from the punctured disc extends
to a conformal harmonic map $\D\to D$.
\end{abstract}

\maketitle

%
%
%
%
\section{Introduction}\label{sec:intro} 

It is classical that a bounded convex domain with smooth boundary in $\R^n$ has a smooth convex 
defining function; see Herbig and McNeal \cite{HerbigMcNeal2012} and the references therein
(in particular, Weinstock \cite[pp.\ 402--403]{Weinstock1975} and
Gilbarg and Trudinger \cite[pp.\ 354--357]{GilbargTrudinger1983}).
In this paper we prove the analogous result for $p$-convex domains for any 
$2\le p<n$. This class of domains, and the associated $p$-plurisubharmonic
functions, are closely related to $p$-dimensional minimal submanifolds; in particular, to minimal
surfaces when $p=2$. They were studied by Harvey and Lawson \cite{HarveyLawson2013IUMJ}; 
see also \cite{HarveyLawson2011ALM,HarveyLawson2012AM} 
and \cite[Sect.\ 8.1]{AlarconForstnericLopez2021}.

Let $\bx=(x_1,\ldots,x_n)$ denote the coordinates on $\R^n$. 
Given a $\Cscr^2$ function $\rho:\Omega\to \R$ on a domain $\Omega\subset \R^n$,
we denote by $\nabla^2\rho(\bx)=\Hess_\rho(\bx)$ the Hessian of $\rho$ at $\bx\in \Omega$, 
i.e., the quadratic form on $T_\bx\R^n=\R^n$
represented by the matrix $\big(\frac{\di^2\rho}{\di x_i\di x_j}(\bx)\big)$.
The trace of $\Hess_\rho(\bx)$ equals the Laplacian of $\rho$ at $\bx$: $\tr\, \Hess_\rho(\bx)=\Delta \rho(\bx)$.
For $2\le p\le n-1$ we denote by $G_p(\R^n)$ the Grassman manifold of $p$-planes in $\R^n$.
Given $\Lambda\in G_p(\R^n)$, we denote by $\tr_\Lambda \Hess_\rho(\bx)\in\R$ 
the trace of the restriction of $\Hess_\rho(\bx)$ to $\Lambda$. 
Our main result is the following.

\begin{theorem}\label{th:main1}
Let $D$ be a bounded domain in $\R^n,\ n\ge 3,$ with boundary of class $\Cscr^{r,\alpha}$
for some $r\in \{2,3,\ldots\}$ and $0<\alpha\le 1$, and let $p\in \{2,\ldots,n-1\}$. If 
the principal interior curvatures $\nu_1(\bx) \le \nu_2(\bx)\le \cdots\le \nu_{n-1}(\bx)$ of $bD$ 
at every point $\bx\in bD$ satisfy
\begin{equation}\label{eq:p-convexity}
	\nu_1(\bx) + \nu_2(\bx) +\cdots + \nu_p(\bx)\ge 0, 
\end{equation}
then there exists a defining function $\rho$ of class $\Cscr^{r,\alpha}$ for $D$ such that 
\begin{equation}\label{eq:p-psh}
	\tr_\Lambda \Hess_\rho(\bx)\ge 0\ \ \text{for all}\ \bx\in \overline D\ \text{and}\ \Lambda\in G_p(\R^n).
\end{equation}
If there are no points $\bx\in bD$ with $\nu_i(\bx)=0$ for $i=1,\ldots,p$
(such a point is called $p$-flat), then $\rho$ can be chosen such that strict inequality holds in 
\eqref{eq:p-psh} for all $\bx\in D$ and $\Lambda\in G_p(\R^n)$.
\end{theorem}

We say that $bD$ is $p$-convex if condition \eqref{eq:p-convexity} holds for every
$\bx\in bD$, and is strongly $p$-convex at $\bx\in bD$ if strict inequality holds in \eqref{eq:p-convexity}. 
Note that $1$-convexity is the usual convexity. In this case, Theorem \ref{th:main1} also holds
and it coincides with the main result of \cite{HerbigMcNeal2012}
(see also \cite[pp.\ 402--403]{Weinstock1975} and \cite[pp.\ 354--357]{GilbargTrudinger1983}).
 
A $\Cscr^2$ function $\rho$ satisfying \eqref{eq:p-psh} is said to be $p$-plurisubharmonic on $\overline D$;
it is strictly $p$-plurisubharmonic at $\bx$ if strict inequality holds at this point.
Recall that $\rho$ is $p$-plurisubharmonic
on a domain $\Omega\subset \R^n$ is and only if the restriction $\rho|_M$ to any 
minimal $p$-dimensional submanifold $M\subset \Omega$ is a subharmonic function
on $M$ in the induced metric, $\Delta_M(\rho|_M)\ge 0$ (cf.\ \cite{HarveyLawson2013IUMJ}
and \cite[Sect.\ 8.1]{AlarconForstnericLopez2021}). 
For affine subspaces this follows by observing that if $\bv_1,\ldots, \bv_p \in \R^n$ is an orthonormal 
basis of $\Lambda\in G_p(\R^n)$ and we set 
$
	\tilde\rho(u_1,\ldots, u_p)=\rho\big(\bx+\sum_{j=1}^p u_j\bv_j\big),
$
then 
\[ 
	\tr_\Lambda \Hess_\rho(\bx) = \Delta \tilde \rho(0).  
\] 
The analogous formula holds on a minimal submanifold (cf.\ \cite[Equation (2.10)]{HarveyLawson2013IUMJ}).
Denoting by $\lambda_1\le\lambda_2\le\cdots\le\lambda_n$ the
eigenvalues of $\Hess_\rho(\bx)$, we have 
\[
	\min_{\Lambda \in G_p(\R^n)} \tr_\Lambda \Hess_\rho(\bx)
	= \lambda_1+\lambda_2+\cdots + \lambda_p.
\]
(See \cite[Lemma 2.5 and Corollary 2.6]{HarveyLawson2013IUMJ} or Lemma \ref{lem:trace} below.)
A domain $D\subset\R^n$ 
is said to be $p$-convex if it admits a strongly $p$-plurisubharmonic exhaustion function
$\rho:D\to\R$, i.e., such that the set $\{\bx \in D:\rho(\bx)\le c\}$ is compact for every $c\in\R$.

Suppose now that $\rho$ is a $\Cscr^2$ defining function for $D$, i.e., $D=\{\rho<0\}$
and $d\rho\ne 0$ on $bD=\{\rho=0\}$. Condition \eqref{eq:p-convexity} is then equivalent to
the following restricted version of \eqref{eq:p-psh}:
\begin{equation}\label{eq:p-convexity2}
	\tr_\Lambda \Hess_\rho(\bx)\ge 0\ \ \text{for every point}\ \bx\in bD\ 
	\text{and $p$-plane}\ \Lambda \subset T_\bx bD.
\end{equation}
Indeed, choosing $\rho$ such that $|\nabla\rho|=1$ on $bD$, the Hessian $\Hess_\rho(\bx)$
at the point $\bx\in bD$ restricted to the tangent space $T_\bx bD$ is the second fundamental form 
of $bD$ at $\bx$ with respect to the inner normal vector $-\nabla\rho(\bx)$
(see \cite[Remark 3.11]{HarveyLawson2013IUMJ}), and 
\[
	\min_{\Lambda\subset T_\bx bD}
	\tr_\Lambda \Hess_\rho(\bx) = \nu_1(\bx)+\nu_2(\bx)+\cdots +\nu_p(\bx).
\]
(The minimum is over all $p$-planes $\Lambda\subset T_\bx bD$. See
\cite[Remark 3.12]{HarveyLawson2013IUMJ} and Lemma \ref{lem:trace}.) 
For a bounded domain $D\subset \R^n$ with smooth boundary we have
(see \cite[Summary 3.16]{HarveyLawson2013IUMJ})   
\[
	\text{$D$ is $p$-convex} \ \Longleftrightarrow\ \text{$bD$ is $p$-convex},
\]
and these are further equivalent to the condition that $-\log\dist(\cdotp,bD)$ is a
$p$-plurisubharmonic function in an interior collar around $bD$. However, the existence 
of a $p$-plurisub\-harmonic defining function is a more convenient property, yielding applications
which are not possible by using unbounded $p$-plurisubharmonic exhaustion functions.

With this terminology, Theorem \ref{th:main1} can be stated as follows.

%
%
\begin{theorem}\label{th:main-simple}
Every bounded $p$-convex domain $D\subset \R^n$ $(2\le p<n)$ with smooth boundary 
admits a smooth defining function $\rho$ which is $p$-plurisubharmonic on $\overline D$.
If in addition $bD$ has no $p$-flat points then $\rho$ can be chosen strongly 
$p$-plurisubharmonic on $D$.
\end{theorem}

Following the terminology introduced in \cite{DrinovecForstneric2016TAMS} and used in 
\cite{AlarconForstnericLopez2021}, $2$-convex domains are called {\em minimally convex} 
and $2$-plurisubharmonic functions are called {\em minimal plurisubharmonic} 
(as they pertain to minimal surfaces). 
Thus, a bounded minimally convex domain $D\subset \R^n$,
$n\ge 3$, with smooth boundary admits a defining function which is minimal plurisubharmonic on 
$\overline D$. This has the following corollary. Let $\D=\{z\in \C:|z|<1\}$ and $\D^*=\D\setminus \{0\}$.

%
%
\begin{corollary}\label{cor:extension}
Let $D$ be a bounded minimally convex domain in $\R^n,\ n\ge 3$, with $\Cscr^r$ boundary 
for some $r>2$. If $M$ is a connected open conformal surface and $f:M\to \R^n$ is a conformal 
harmonic map with $f(M)\subset \overline D$, then either $f(M)\subset D$ or $f(M)\subset bD$. 
If $K\subset M$ is a removable set for bounded harmonic functions, then 
every conformal harmonic map $M\setminus K\to D$ extends to a conformal harmonic map $M\to D$. 
In particular, every conformal harmonic map $\D^*\to D$ from the punctured disc  
extends to a conformal harmonic map $\D\to D$. 
\end{corollary}

\begin{proof}
Let $\rho$ be a defining function for $D$ such that \eqref{eq:p-psh} holds for $p=2$. 
The first claim follows from the maximum principle applied to the subharmonic function 
$\rho\circ f:M\to (-\infty,0]$, and the second claim is an immediate consequence. 
The last claim follows by observing that every bounded harmonic function on $\D^*$
extends to a harmonic function on $\D$.
\end{proof}

%
%
\begin{remark}
If $bD$ fails to be minimally convex at a point $\bp\in bD$, then there is an embedded minimal 
disc $f:\D\to D\cup \{\bp\}$ with $f(0)=\bp$ and $f(\D\setminus \{0\})\subset D$ 
(see \cite[Lemma 3.13]{HarveyLawson2013IUMJ}). Hence, the conclusion of 
Corollary \ref{cor:extension} fails in this case. 
\end{remark}

%
%
\begin{remark}
Corollary \ref{cor:extension} generalizes the classical fact that two immersed 
minimal surfaces in $\R^3$ which touch at a point but locally at that point lie on one side of one another  
must coincide. In fact, assume that $D\subset\R^3$ is a bounded domain whose smooth 
boundary contains a domain $\Sigma$ in an immersed minimal surface $\wt \Sigma\subset\R^3$
(hence, $\Sigma$ is embedded),  
and $bD$ is strongly minimally convex at each point $\bp\in bD\setminus \overline \Sigma$. 
Then, the domain $D$ is minimally convex. Let $\rho$ be a minimal plurisubharmonic defining function for $D$
provided by Theorem \ref{th:main1}. Given a connected open conformal surface $M$ and a 
nonconstant conformal harmonic map $f:M\to \overline D$, the maximum principle for the 
subharmonic function $\rho\circ f:M\to (-\infty,0]$ shows that either $f(M)\subset D$ 
or $f(M)\subset \Sigma$.
\end{remark}

%
%
We shall see that a suitable convexification of the signed distance function to $bD$,
\begin{equation}\label{eq:delta}
	\delta(\bx) = \dist(\bx, D) - \dist(\bx,\R^n\setminus D), \quad \bx\in\R^n
\end{equation}
satisfies the conclusion Theorem \ref{th:main1} in an interior collar around $bD$.
Assuming that $bD$ is of H\"older class $\Cscr^{r,\alpha}$ for some 
$r\in \{2,3,\ldots\}$ and $0<\alpha\le 1$, $\delta$ is of the same class $\Cscr^{r,\alpha}$ 
in a neighbourhood $V\subset \R^n$ of $bD$ (see Li and Nirenberg \cite{LiNirenberg2005}). 
This gives the following corollary.

%
%
\begin{corollary}
Let $D$ be a bounded $p$-convex domain with $\Cscr^r$ boundary in $\R^n$
for some $2\le p<n$ and real number $r>2$. Then every domain 
$D_t = \{\bz\in D:\delta(\bz)<t\}$ for $t<0$ close to $0$ is $p$-convex, 
and it is strongly $p$-convex if $bD$ does not contain any $p$-flat points.
\end{corollary}

On a domain $\Omega\subset\C^n$ and for a complex line 
$\Lambda\subset \C^n$, $\tr_\Lambda \Hess_\rho(\bx)$ is the Levi form of $\rho$ at $\bx\in\Omega$ 
on any unit vector $\xi\in \Lambda$. It follows that every minimal plurisubharmonic
function on a domain in $\C^n$ is also plurisubharmonic in the usual sense of complex analysis, 
and every 2-convex domain $D$ in $\C^n$ is pseudoconvex; the converse fails.
Assuming that $bD$ is of class $\Cscr^r$ for some $r>2$, 
Levi pseudoconvexity of $bD$ is characterized by nonnegativity of the Hessian of $\delta$ applied 
in complex tangent directions to $bD$. From the differential geometric viewpoint, Levi pseudoconvexity
says that the mean sectional curvature of $bD$ in complex tangent directions is nonnegative 
(see \eqref{eq:HDelta} and \eqref{eq:Levi}). 
This leads to an observation regarding strong pseudoconvexity of domains 
$D_t = \{\bz\in D:\delta(\bz)<t\}$ for $t<0$ close to $0$; see Theorem \ref{th:main3}. 

%
%
%
%
\section{Proof of Theorem \ref{th:main1}}\label{sec:main1}
Let $Q$ be a real symmetric $n\times n$ matrix for $n\ge 2$ and 
$\Qcal$ be the associated quadratic form 
$
	\Qcal(\bx)=Q\bx\,\cdotp \bx,\ \bx\in\R^n. 
$
(The dot denotes the Euclidean inner product.) 
Let $1\le p < n$. Given a $p$-plane $\Lambda\in G_p(\R^n)$, we denote by $\tr_\Lambda \Qcal$ 
the trace of the restricted quadratic form $\Qcal|_\Lambda$.  The following result 
is \cite[Lemma 2.5 and Corollary 2.6]{HarveyLawson2013IUMJ};
we include a simple proof.

%
%
\begin{lemma}\label{lem:trace}
Let $\lambda_1\le \lambda_2\le\cdots\le \lambda_n$ be the eigenvalues of $Q$. 
For every $\Lambda\in G_p(\R^n)$ we have that 
$
	\tr_\Lambda \Qcal \ge \lambda_1+\lambda_2+\cdots+\lambda_p,
$
with equality if and only if $\Lambda$ admits an orthonormal basis $\bv_1,\ldots,\bv_p$
such that $Q\bv_i=\lambda_i\bv_i$ for $i=1,\ldots,p$. 
\end{lemma}

\begin{proof}
This is obvious for $p=1$, and we proceed by induction on $p$. 
We may assume that in the standard basis $\be_1,\ldots,\be_n$ of $\R^n$ the matrix 
$Q$ is diagonal with $Q\be_i=\lambda_i \be_i$ for $i=1,\ldots,n$. 
The subspace $\Lambda'=\Lambda\cap (\{0\}\times\R^{n-1})$ has dimension at least $p-1$.
Pick an orthonormal basis $\bv_1,\bv_2,\ldots,\bv_n$ of $\Lambda$ with 
$\bv_2,\ldots,\bv_n \in \Lambda'$. Then, 
$
	\tr_\Lambda \Qcal = Q\bv_1\,\cdotp \bv_1 + \sum_{j=2}^p Q\bv_j\,\cdotp \bv_j.
$
Clearly, $Q\bv_1\,\cdotp \bv_1 \ge \lambda_1$ 
with equality if and only if $Q\bv_1=\lambda_1\bv_1$. Since $\Lambda'':=\span\{\bv_2,\ldots,\bv_p\}$ 
is contained in the invariant subspace $\{0\}\times\R^{n-1}$ of $Q$ with eigenvalues 
$\lambda_2\le \cdots \le \lambda_{n}$, the inductive hypothesis
gives $\sum_{j=2}^p Q\bv_j\,\cdotp \bv_j\ge  \lambda_2+\cdots +\lambda_p$,
with equality if and only if $\Lambda''$ is spanned by eigenvectors for eigenvalues
$\lambda_2,\ldots,\lambda_p$. This completes the induction step. 
\end{proof}

%
%
\begin{proof}[Proof of Theorem \ref{th:main1}]
Let $D\subset \R^n$ be a bounded domain with $\Cscr^{r,\alpha}$ boundary
for some integer $r\ge 2$ and real number $0<\alpha\le 1$. 
Denote by $\delta$ the signed distance function \eqref{eq:delta} from $bD$. 
By Li and Nirenberg \cite{LiNirenberg2005}, $\delta$ is of the same class $\Cscr^{r,\alpha}$
in a collar neighbourhood $V\subset \R^n$ of $bD$. 
We recall some further properties of $\delta$, referring to 
Bellettini \cite[Theorem 1.18, p.\ 14]{Bellettini2013} 
and Gilbarg and Trudinger \cite[Section 14.6]{GilbargTrudinger1983}.
Shrinking $V$ around $bD$ if necessary, there is a $\Cscr^r$ projection 
$\pi:V\to bD$ such that for any $\bx\in V$, $\bp=\pi(\bx)\in bD$ is the unique nearest 
point to $\bx$ on $bD$. The gradient $\nabla \delta$ has constant norm $|\nabla\delta|=1$ on $V$, 
and it has constant value on the intersection of $V$ with the normal line 
$N_\bp=\bp+\R\,\cdotp \nabla\delta(\bp)$ at $\bp\in bD$.  
By a translation and an orthogonal rotation we may assume that $\bp=\zero,\ T_\zero bD=\{x_n=0\}$, 
and the standard basis vectors $\be_1,\be_2,\ldots,\be_n=\nabla\delta(\zero)$ of $\R^n$ 
diagonalize the matrix $H$ of $\Hess_\delta(\zero)$:
\[
	H\be_i=\nu_i \be_i,\quad\ i=1,\ldots,n-1;\quad H\be_n=0.
\]
The restriction of $\Hess_\delta(\zero)$ to the tangent space $T_\zero bD=\R^{n-1}\times \{0\}$ 
is the second fundamental form of $bD$ at $\zero$ and $\nu_1,\ldots,\nu_{n-1}$
are the principal curvatures of $bD$ at $\zero$ from the side $-\be_n=-\nabla\delta(\zero)$. 
By reordering we may assume that $\nu_1\le \nu_2\le \cdots\le \nu_{n-1}$.
For any point $\bx=(0,\ldots,0,t)\in N_\zero \cap V$ the basis $\be_1,\ldots,\be_n$ 
also diagonalizes $\Hess_\delta(\bx)$, with the eigenvalues 
\begin{equation}\label{eq:nux}
	\nu_i(\bx)=\frac{\nu_i}{1 + t \nu_i},\quad\ i=1,\ldots,n.
\end{equation}
In particular, $\nu_n(\bx)=0$. Note that for $t<0$ (which corresponds to 
$\bx\in N_\zero \cap V \cap D$) we have $\nu_i(\bx)\ge \nu_i$, while 
for $t>0$ (i.e., $\bx\in N_\zero \cap V \setminus \overline D$) we have 
$\nu_i(\bx)\le \nu_i$. In both cases equality holds if and only if $\nu_i=0$. 
Furthermore, $\nu_1(\bx) \le \cdots \le \nu_{n-1}(\bx)$ for all $\bx \in N_\zero\cap V$. 

Assume now that $bD$ is $p$-convex at $\zero$ in the sense that condition \eqref{eq:p-convexity}
holds, i.e., 
\[	
	\nu_1+\nu_2+\cdots+\nu_p\ge 0.
\] 
By Lemma \ref{lem:trace} applied with $\Qcal=\Hess_\delta(\zero)$ 
restricted to $T_\zero bD$, this is equivalent to 
\[
	\tr_\Lambda \Hess_\delta(\zero) \ge 0\ \ \text{for every $p$-plane $\Lambda\subset T_\zero bD$}.
\]
Consider the family of domains $D_t=\{\delta<t\}$ for $t$ near $0$. 
As $t$ increases, the domains $D_t$ strictly increase, and $D_0=D$. The tangent space
to $bD_t=\{\delta=t\}$ at $\bx=(0,\ldots,0,t)\in bD_t$ is $\R^{n-1}\times \{0\}$, 
$\be_n=\nabla\delta(\bx)$ is the unit outer normal vector at $\bx$, and the numbers 
$\nu_i(\bx)$ \eqref{eq:nux} for $i=1,\ldots,n-1$ are the principal normal curvatures of $bD_t$ at $\bx$. 
From what has been said, we see that at $\bx=(0,\ldots,0,t)\in bD_t$ with $t<0$ we have that
\[
	\nu_1(\bx)+\nu_2(\bx)+\cdots+\nu_p(\bx)\ge \nu_1+\nu_2+\cdots+\nu_p\ge 0
\]
with equality if and only if $\nu_1=\nu_2=\cdots=\nu_p=0$, i.e., the point $\zero\in bD$
is $p$-flat. Since this analysis holds at any point $\bp\in bD$, we conclude that for $t<0$
close to $0$ the boundary $bD_t$ is $p$-convex, and it is strongly $p$-convex if and only if
$bD$ does not contain any $p$-flat points.

It remains to find a defining function for $D$, having the domains $D_t$ as sublevel 
sets, which is $p$-plurisubharmonic on $\overline D$, and is strongly $p$-plurisubharmonic 
on $D$ provided that $bD$ has no $p$-flat points. 
Consider the function $h\circ \delta:V\to\R$, where $h$ is a smooth convex increasing 
function on $\R$ with $h(0)=0$, $\dot h(0)=1$, and $\ddot h(0)=a>0$ (to be determined). Then, 
\[
	\Hess_{h\circ \delta} = 
	(\dot h\circ \delta) \Hess_\delta + (\ddot h\circ \delta) \nabla \delta\, \cdotp (\nabla \delta)^T.
\]
Assume that for $\bp\in bD$ the orthonormal vectors $\bv_1,\ldots,\bv_{n-1},\bv_n$ 
diagonalize $\Hess_\delta(\bp)$, where $T_\bp D=\span\{\bv_1,\ldots,\bv_{n-1}\}$ and 
$\bv_n=\nabla\delta(\bp)$. Then the vectors $\bv_1,\ldots,\bv_{n-1}$ lie in the kernel of the matrix 
$\nabla \delta(\bp)\, \cdotp \nabla \delta(\bp)^T$, while $\bv_n$ is an eigenvector with eigenvalue $1$.  
Hence, the basis $\bv_1,\ldots,\bv_n$ diagonalizes 
$\Hess_{h\circ \delta}(\bx)$ at every point $\bx=\bp+ \delta(\bx)\bv_n\in N_\bp\cap V$, 
the eigenvalues corresponding to $\bv_1,\ldots,\bv_{n-1}$  
get multiplied by the number $\dot h(d(\bx))$, which is close to $1$ if $\bx$ is close to $\bp$, 
and the eigenvalue in the normal direction $\bv_n=\nabla \delta(\bx)$ is $\ddot h(\delta(\bx))\approx a$. 
By choosing $a$ such that $a+\lambda_1(\bp)>0$ for all $\bp\in bD$ 
and shrinking the neighbourhood $V$ around $bD$, 
the function $h\circ \delta$ is $p$-plurisubharmonic on $\overline D\cap V$,
and it is strongly $p$-plurisubharmonic on $D\cap V$ if $bD$ has no $p$-flat points.
Choose $c<0$ such that $\{c\le \delta\le 0\}\subset \overline D\cap V$. 
Let $\phi:\R\to\R$ be a convex increasing function such that $\phi(t)=2c/3$ for $t\le c$ and 
$\phi(t)=t$ for $t\ge c/2$. Then, $\rho=\phi\circ h\circ \delta$ is a well-defined function 
on $\overline D\cup V$ which equals $h\circ \delta$ outside $D_{c/2}$ and is 
$p$-plurisubharmonic on $\overline D$.

Assume now that $h\circ \delta$ is strongly $p$-plurisubharmonic on a collar 
$\{c\le \delta<0\}$. Then, $\rho=\phi\circ h\circ \delta$ is strongly $p$-plurisubharmonic
on $\{c/2 \le \delta<0\}\subset D\cap V$. We take 
\[
	\tilde \rho(\bx) = \rho(\bx) + \epsilon \chi(\bx)\cdotp |\bx|^2, 
\]
where $\chi:\R^n\to [0,1]$ is a smooth function which equals $1$ on $\overline D_{c/3}$
and has compact support contained in $D$. The support of the differential of $\chi$ is then compact
and contained in the set $\{c/2< \delta<0\}$ where $\rho$ is strongly $p$-plurisubharmonic. 
Note that $|\bx|^2=\sum_{i=1}^n x_i^2$ is strongly $p$-plurisubharmonic on $\R^n$. 
Choosing $\epsilon>0$ small enough therefore ensures that $\tilde \rho$ is strongly
plurisubharmonic on $D$. This proves Theorem \ref{th:main1}. 
\end{proof}

%
%
Let us look more closely at the case $p=2$ of Theorem \ref{th:main1}, which is of interest for the  
theory of $2$-dimensional minimal surfaces. In particular, every $2$-convex domain in $\R^n$, $n\ge 3$,
admits plenty of proper minimal surfaces parameterized by conformal harmonic immersions from 
any given bordered Riemann surface (see \cite[Theorems 8.3.1 and 8.3.11]{AlarconForstnericLopez2021}). 
In this case, condition \eqref{eq:p-convexity2} (which is equivalent to condition \eqref{eq:p-convexity}
in Theorem \ref{th:main1}) can be expressed in more intrinsic differential geometric terms, thereby 
exposing a connection to Levi pseudoconvexity studied in complex analysis; see the following section 
for the latter. 

Let $D$ be a smoothly bounded domain in $\R^n$, $n\ge 3$.
Given a point $\bp\in bD$ and a $2$-plane $\Lambda\subset T_\bp bD$,  
let $S(\bp,\Lambda)$ be the surface obtained by intersecting $bD$ with the affine 3-plane 
$\bp\in \Sigma_\bp\cong \R^3$ spanned by $\Lambda$ and the normal vector to $bD$ at $\bp$. 
Let $\Qcal_\bp$ denote the second fundamental form of $bD$ at $\bp$.
The principal curvatures $\lambda_1(\bp,\Lambda),\lambda_2(\bp,\Lambda)$ of 
$S(\bp,\Lambda)$ at $\bp$ are the eigenvalues of the restricted quadratic form 
$\Qcal_\bp|_{\Lambda}$. Their sum 
\begin{equation}\label{eq:Hp}
	H(\bp,\Lambda) =  \lambda_1(\bp,\Lambda)+\lambda_2(\bp,\Lambda) 
\end{equation}
is the mean curvature of $S(\bp,\Lambda)$ at $\bp$, and their product
\begin{equation}\label{eq:Kp}
	K(\bp,\Lambda) = \lambda_1(\bp,\Lambda)\,\cdotp \lambda_2(\bp,\Lambda)
\end{equation}
is the Gaussian curvature of $S(\bp,\Lambda)$ at $\bp$. These numbers are,
respectively, the sectional mean curvature and the sectional Gaussian curvature of
$bD$ at $\bp$ on the 2-plane $\Lambda$. While $H$ reflects the way how $bD$ sits in 
$\R^n$, $K$ is an intrinsic quantity depending only on the induced metric on $bD$.
If $\rho$ is a defining function for $D$ such that $|\nabla\rho(\bp)|=1$, 
then $\Qcal_\bp$ equals the Hessian $\Hess_\rho(\bp)$ restricted to $T_\bp bD$,
and hence
\begin{equation}\label{eq:HDelta}
	H(\bp,\Lambda) = \tr_\Lambda \Hess_\rho(\bp) = \Delta(\rho|_{\bp+\Lambda})(\bp).
\end{equation}
If $H(\bp,\Lambda)=0$ then $K(\bp,\Lambda) = -\lambda_1(\bp,\Lambda)^2\le 0$,
and $K(\bp,\Lambda)=0$ if and only if $\Qcal_\bp|_\Lambda$ vanishes.
Theorem \ref{th:main1} for $p=2$ can now be stated as follows.

%
%
\begin{theorem}\label{th:main2}
Let $D$ be a bounded domain in $\R^n,\ n\ge 3$, with $\Cscr^{r,\alpha}$ boundary
for some $r\ge 2$ and $0<\alpha\le 1$. If 
\begin{equation}\label{eq:MCboundary}
	H(\bp,\Lambda)\ge 0\ \ 
	\text{for every point $\bp\in bD$ and 2-plane $\Lambda\subset T_\bp bD$},
\end{equation} 
then $D$ admits a $\Cscr^{r,\alpha}$ defining function $\rho$ satisfying \eqref{eq:p-psh}
with $p=2$ (i.e., $\rho$ is minimal plurisubharmonic on $\overline D$). 
If in addition there is no $(\bp,\Lambda)$ as above such that 
$H(\bp,\Lambda)= K(\bp,\Lambda)=0$, then $\rho$ can be chosen strongly 
minimal plurisubharmonic on $D$.
\end{theorem}

%
%
%
%
\section{A remark on Levi pseudoconvex domains}\label{sec:Levi}
For a smoothly bounded domain $D$ in a complex Euclidean space $\C^n, n\ge 2$, the $2$-convexity 
condition \eqref{eq:p-convexity2} applied only to complex lines $\Lambda\subset T_\bp bD$ 
is precisely the Levi pseudoconvexity condition.  
Indeed, if $\bv\in \Lambda=\C \bv$ is a unit vector then 
\begin{equation}\label{eq:Levi}
	(dd^c \rho)_\bp (\bv,J\bv)=\Delta(\rho|_{\bp+\Lambda})(\bp) = 
	\tr_\Lambda \Hess_\rho(\bp) = H(\bp,\Lambda),
\end{equation}
where $J$ is the standard almost complex structure operator on $\C^n$.

It is well known that if $D$ is a bounded pseudoconvex domain in $\C^n$ then the function
$-\log \dist(\cdotp,bD)$ is plurisubharmonic on $D$, and hence the domains
$D_t=\{z\in D:\dist(z,bD)>t\}=\{\delta<-t\}$ for $t>0$ are pseudoconvex. This is a stronger result
than we can get for minimally convex domains, and it does not require any smoothness of $bD$.
However, the following observation related to Theorem \ref{th:main2} seems worthwhile recording. 

%
%
\begin{theorem}\label{th:main3}
Let $D$ be a bounded pseudoconvex domain in $\C^n,\ n\ge 2,$ 
with $\Cscr^{r,\alpha}$ boundary for some $r\ge 2$ and $0<\alpha\le 1$, and  
let $\delta$ be given by \eqref{eq:delta}. 
If the sectional curvature $K(\bp,\Lambda)$ \eqref{eq:Kp} is nonzero
for every point $\bp\in bD$ and complex line $\Lambda\subset T_\bp bD$ on which the 
Levi form \eqref{eq:Levi} vanishes, then the domains $D_t=\{\delta<t\}$ for $t<0$ 
close to $0$ are strongly pseudoconvex.
\end{theorem}


%
%
\begin{proof} 
Let $\delta$ denote the signed distance function \eqref{eq:delta} to $bD$.
Fix a point $\bp\in bD$ and let $\bv=\nabla \delta(\bp)$. 
Let $J$ denote the standard almost complex structure on 
$\C^n$, which corresponds to multiplication by $\imath$ on complex tangent vectors.
The tangent space to $bD$ at $\bp$ is an orthogonal direct
sum $T_\bp bD= T_\bp^\C bD\oplus E$, where $\Sigma:=T_\bp^\C bD=T_\bp bD\cap J(T_\bp bD)$
is the maximal complex subspace of $T_\bp bD$ (a complex hyperplane) 
and $E$ is the real line orthogonal to $T_\bp^\C bD$. Note that $J(E)=\R \bv$ 
is the normal line to $bD$ at $\bp$. The Levi pseudoconvexity condition is that 
$H(\bp,\Lambda) = \tr_\Lambda \Hess_\delta(\bp) \ge 0$ for every complex line 
$\Lambda\subset T_\bp^\C bD$ (see \eqref{eq:Levi}). 

Let $D_t=\{\delta<t\}$ for $t<0$ close to $0$ and $\bx=\bp+t\bv \in bD_t$. 
Then, $T_\bx bD_t$ is the orthogonal complement to $\bv$, and hence $T_\bx bD_t=\Sigma\oplus E$, i.e., 
the same orthogonal decomposition of $T_\bx bD_t$ holds along the line $N_\bp=\bp+\R\bv$. 
To prove the theorem, it suffices to show that $\tr_\Lambda \Hess_\delta(\bx) \ge H(\bp,\Lambda)$, 
with equality if and only if $H(\bp,\Lambda)=K(\bp,\Lambda)=0$. 
This will imply that $bD_t$ for $t<0$ close to $0$ is Levi pseudoconvex, and it is strongly Levi pseudoconvex if and only if there are no complex lines $\Lambda\subset T_\bp bD$, $\bp\in bD$, 
with $H(\bp,\Lambda)=K(\bp,\Lambda)=0$.

Let $\bv_1,\dots,\bv_{n-1}\in T_\bp bD$ be orthonormal eigenvectors of $\Hess_\delta(\bp)$ 
with the eigenvalues $\nu_1\le \nu_2\le \cdots\le \nu_{n-1}$. 
Fix a complex line $\Lambda\subset T_\bp bD$. For any orthonormal basis 
$\xi=\sum_{i=1}^{n-1} \xi_i \bv_i$ and $\eta=\sum_{i=1}^{n-1} \eta_i \bv_i$ for $\Lambda$
we have that 
\begin{equation}\label{eq:decrease}
	\tr_\Lambda \Hess_\delta(\bx) = \sum_{i=1}^{n-1} \nu_i(\bx)(\xi_i^2+\eta_i^2)
	\ge \sum_{i=1}^{n-1} \nu_i(\xi_i^2+\eta_i^2)
	= \tr_\Lambda \Hess_\delta(\bp) = H(\bp,\Lambda),
\end{equation}
where we used $\nu_i(\bx)\ge \nu_i$ (see \eqref{eq:nux}).
Equality holds if and only if $\xi_i^2+\eta_i^2=0$ whenever 
$\nu_i\ne 0$ (since for such $i$ we have $\nu_i(\bx)>\nu_i$).
Clearly this holds if and only if the vectors $\xi$ and $\eta$ 
are contained in the kernel of $\Hess_\delta(\bp)$. Since this argument holds for every
orthonormal basis of $\Lambda$, we infer that equality holds in \eqref{eq:decrease}
if and only if $H(\bp,\Lambda)=K(\bp,\Lambda)=0$.
\end{proof}

%
%
\begin{remark}\label{rem:pshdef}
(A) Regarding plurisubharmonic defining functions,
there is a major difference between the minimal case (cf.\ Theorem \ref{th:main1}) 
and the complex (Levi) case. It has been known since the 1970s
through examples of Diederich and Forn{\ae}ss \cite{DiederichFornaess1977MA}
and Forn{\ae}ss \cite[Example 5]{Fornaess1979} that there exist bounded 
weakly pseudoconvex domains $D\subset\C^2$ with smooth real analytic boundaries
which do not admit a defining function that is plurisubharmonic on $\overline D$,
in the sense that its Levi form is nonnegative at every point of $\overline D$ in all
complex directions. Furthermore, on a pseudoconvex domain admitting a plurisubharmonic defining
function the $\dibar$--Neumann problem is exactly regular; see Boas and Straube
\cite{BoasStraube1991}. This is not true on all weakly pseudoconvex domains; 
it fails in particular on certain worm domains according to 
Barrett \cite{Barrett1992} and Christ \cite{Christ1996}.
Hence, Theorem \ref{th:main2} has no analogue in the complex case.
The reason is that Levi pseudoconvexity provides much less information 
since it only pertains to complex lines, as opposed to all $2$-planes tangent to the boundary. 
In particular, it does not give any condition in the direction 
normal to the maximal complex subspace of $T_\bp bD$. 

(B) On the other hand, it was shown by Diederich and Forn\ae ss in \cite{DiederichFornaess1978MM}
that locally near any boundary point of a smoothly bounded weakly pseudoconvex domain 
in $\C^n$ one can make a local holomorphic change of coordinates so that the level 
sets of the signed distance function are strongly pseudoconvex in the interior 
of the new local domain. The point is that, while a biholomorphic map 
preserves the sign of the Levi form on complex directions (and hence of the function $H(\bp,\Lambda)$
in \eqref{eq:Levi}), the sectional Gaussian curvature $K(\bp,\Lambda)$ \eqref{eq:Kp} may very well change 
from zero to a nonzero value. If we ensure that this happens on all complex tangent directions
on which the Levi form vanishes, then Theorem \ref{th:main3} applies.
\end{remark}

%
%

%
%
%
\smallskip
\noindent {\bf Acknowledgement.} 
I wish to thank John Erik Forn\ae ss for the information discussed in Remark \ref{rem:pshdef} (B)
and for the relevant reference \cite{DiederichFornaess1978MM}. I also thank the referee for 
remarks which helped me to improve the presentation.


\end{document}